\documentclass{amsart}
\usepackage[english]{babel}
\usepackage[utf8]{inputenc}
\usepackage{amsmath, amssymb,epic,graphicx,mathrsfs,enumerate}
\usepackage[all]{xy}
\usepackage{color}
\usepackage{comment}
\usepackage{enumitem}

\usepackage[]{frontespizio}
\usepackage{amssymb}
\usepackage{amsmath}
\usepackage{amsthm}
\usepackage{graphicx}
\usepackage[colorlinks=true]{hyperref}
\definecolor{air}{rgb}{0.36, 0.54, 0.66}
\hypersetup{
	citecolor=air,
	linkcolor=air}
\usepackage{amscd}
\usepackage{enumerate}
\usepackage{mathrsfs}

\usepackage{amsmath}
\usepackage{amsthm}
\usepackage{amssymb}
\usepackage{latexsym}
\usepackage{epsfig}

\DeclareMathOperator{\core}{Core} 
\DeclareMathOperator{\frat}{Frat} 
 
\DeclareMathOperator{\aut}{Aut}

\DeclareMathOperator{\soc}{soc}

\newcommand{\sym}{\mathrm{Sym}}

\DeclareMathOperator{\End}{End}

\newcommand{\gen}[1]{\left\langle#1\right\rangle}

\newcommand{\st}{such that }
\newcommand{\ifa}{if and only if }

\newtheorem{thm}{Theorem}
\newtheorem{cor}[thm]{Corollary}
 \newtheorem{lemma}[thm]{Lemma}
\newtheorem{prop}[thm]{Proposition} 
 \newtheorem{defn}[thm]{Definition}
 \newtheorem{con}[thm]{Conjecture}

\numberwithin{equation}{section}

\renewcommand{\footnote}{\endnote}
\newcommand{\ignore}[1]{}\makeglossary

\begin{document}
	\bibliographystyle{amsplain}

\title{On the connectivity of the generating and rank graphs of finite groups}

\author{Andrea Lucchini}
\address{A. Lucchini, Universit\`a di Padova, Dipartimento di Matematica ``Tullio Levi-Civita'', Via Trieste 63, 35121 Padova, Italy}
\email{lucchini@math.unipd.it}

\author{Daniele Nemmi}
\address{D. Nemmi, Universit\`a di Padova, Dipartimento di Matematica ``Tullio Levi-Civita'', Via Trieste 63, 35121 Padova, Italy}
\email{dnemmi@math.unipd.it}

\begin{abstract}
	The generating graph encodes how generating pairs are spread among the elements of a group. For more than ten years it has been conjectured that this graph is connected for every finite group. In this paper, we give evidence supporting this conjecture: we prove that it holds for all but a finite number of almost simple groups and give a reduction to groups without non-trivial soluble normal subgroups. Let $d(G)$ be the minimal cardinality of a generating set for $G$. When $d(G)\geq3$, the generating graph is empty and the conjecture is trivially true. We consider it in the more general setting of the rank graph, which encodes how pairs of elements belonging to generating sets of minimal cardinality spread among the elements of a group. It carries information even when $d(G)\geq3$ and corresponds to the generating graph when $d(G)=2$. We prove that it is connected whenever $d(G)\geq3$, giving tools and ideas that may be used to address the original conjecture.
	
\end{abstract}
\thanks{Project funded by the EuropeanUnion – NextGenerationEU under the National Recovery and Resilience Plan (NRRP), Mission 4 Component 2 Investment 1.1 - Call PRIN 2022 No. 104 of February 2, 2022 of Italian Ministry of University and Research; Project 2022PSTWLB (subject area: PE - Physical Sciences and Engineering) " Group Theory and Applications"}
\maketitle

\hbox{}

\section{Introduction}

The study of generating properties of finite groups is a very old subject of investigation and plays a central role in many aspects of group theory, both from a theoretic and a computational point of view. A milestone in this topic is a paper of Steinberg \cite{Stein} who, in 1962, proved that the finite simple groups known at the time (almost all) can be generated by two elements. 	This strong result led to many interesting questions which have been the focus of intensive research in the recent years. In 2000, Guralnick and Kantor \cite{GK} proved a stronger version of this result: i.e. in a finite simple group, every element different from the identity is contained in a generating pair and in 2021, Burness, Guralnick and Harper \cite{spread} established a conjecture, remained open for many years, which states that in a finite group $G$, any non-identity element belongs to a generating pair if and only if every proper quotient of $G$ is cyclic.
Another notable achievement of the investigation on generating properties of simple groups is that if $G$ is a finite simple group, the probability that two random elements of $G$ are a generating pair approaches $1$ if $|G|\rightarrow\infty$, see \cite{dixon}, \cite{KL}, \cite{LibSha}.
The research on the topic therefore shows an abundance of generating pairs in simple groups which has applications in many areas of group theory. It turned out to be a successful strategy to encode in a graph how these pairs spread within the group. This approach was used implicitly for example in \cite{LibSha2}, and presented explicitly in 2008, when the first author and Mar\'oti introduced the \emph{generating graph} \cite{LM1}. Given a non-cyclic finite group $G$, we consider the graph $\bar\Gamma(G)$ whose vertex set is $G$ and two elements $g,h\in G$ are joined if $G=\gen{g,h}$ and then define the generating graph $\bar\Delta(G)$ of $G$ as the graph induced by the non-isolated vertices of $\bar\Gamma(G)$. Several results are known about this graph: for survey references the reader can see \cite{bur_survey} and \cite{har_survey}. 

Along the years, another indicator of the abundance of generating pairs in groups has been noticed: it turned out that the generating graph is connected for many classes of groups: for example in 2013 the first author and Crestani \cite{CL3} proved that $\bar\Delta(G)$ is connected for every soluble group and in \cite{CL2} they proved that it is connected for characteristically simple groups, moreover the first author \cite{Luc17} shows that if $G$ is soluble, then ${\rm diam}(\bar\Delta(G))\leq 3$. In \cite{spread} the authors generalise previous results about simple groups by showing that if every proper quotient of $G$ is cyclic, then $\bar\Delta(G)$ is connected with diameter $2$. 

Therefore the following conjecture has naturally appeared in the area:

\begin{con}\label{conj}
	$\bar\Delta(G)$ is connected for every finite group $G$.
\end{con}

All the results listed above support Conjecture \ref{conj} and it seems that in addition to the fact
that $\bar\Delta(G)$ is connected, its diameter is bounded. However this last fact is not true:
in \cite[Theorem 5.4]{CL2} a family of groups in which
this diameter is unbounded is presented.

The aim of this paper is to give more evidence supporting Conjecture \ref{conj}. The first result we present is the asymptotic validity of the conjecture for almost simple groups: in particular we have

\begin{thm}\label{almost}
	Let $G$ be a 2-generated almost simple group with socle $S$. If $|S|$ is large enough, then $\bar\Delta(G)$ is connected.
\end{thm}




When a group $G$ cannot be generated by two elements, the generating graph of $G$ is the empty graph which is trivially connected. In what follows we give a more general statement of Conjecture \ref{conj} which includes (in a non-trivial way) also groups which are not $2$-generated. The second main achievement of the paper is the proof of this new conjecture for groups with $d(G)\geq3$, being $d(G)$ the minimum number of elements needed to generate $G$. In doing so, we present additional sufficient conditions for the generating graph to be connected and give the following reduction of Conjecture \ref{conj} which is straightforward from Lemma \ref{norsol}:
\begin{cor}
	$\bar\Delta(G)$ is connected for every finite group $G$ if and only if $\bar\Delta(G)$ is connected for every finite group $G$ without non-trivial normal soluble subgroups.
\end{cor}

We now present the new setting in which we addressed Conjecture \ref{conj}.
Several possibilities of generalisations of the concept of
the generating graph can be considered for groups which are not 2-generated. Here we see the generating graph of 2-generated
non-cyclic groups as a graph on the elements of the group in which two of them
are joined if they belong to a generating set of $2 = d ( G )$ elements and extend this
definition in the following way.


\begin{defn}
	Given a non-cyclic finite group $G$ and an integer $d\geq d(G)$, denote by $\Gamma_d(G)$ the graph whose vertices are the elements of $G$ and where two different vertices $x$ and $y$ are adjacent if there exists a generating set of $G$ containing $x$ and $y$ and with cardinality $d.$
	Let $\Delta_d(G)$ be the graph obtained from  $\Gamma_d(G)$ by removing the isolated vertices.  We call
	$\Delta(G)=\Delta_{d(G)}(G)$  the rank graph of $G$.
\end{defn}

Notice that when $d(G)=2$, $\Delta(G)$ corresponds to the generating graph. 
In this new setting we present our main result supporting Conjecture \ref{conj}:

\begin{thm}\label{main_conn_rank}
	Let $G$ be  a finite group.
	Then $\Delta_d(G)$ is connected for every $d\geq \max\{3,d(G)\}$, in particular the rank graph $\Delta(G)$ is connected whenever $d(G)\geq3$.
\end{thm}

Questions about connectivity of generation graphs of finite groups are often related to properties of almost simple groups. This case is not an exception, in fact the proof of Theorem \ref{main_conn_rank} relies on a result on the structure of almost simple groups proved in a recent joint work with Costantini \cite{cln}:

\begin{thm}\label{cln}
Let $G$ be a finite monolithic group and suppose that $N=\soc(G)$ is non-abelian. If $a, b$ are two elements of $G$ with the property that $[a,b]\in N,$ then there exist $n, m \in N$ such that $[an,bm]=1.$	
\end{thm}


The proof of Theorem \ref{main_conn_rank} may also offer some tools which could be helpful to address Conjecture \ref{conj}, in fact part of the proof is independent of the value of $d$.

We are now ready to present our results: in Section \ref{sec:almost} we prove Theorem \ref{almost} and in Section \ref{sec:conn_rank} we prove Theorem \ref{main_conn_rank}.

\section{The generating graph of almost simple groups}\label{sec:almost}

Let $G$ be a 2-generated almost simple group and let $S=\soc(G).$ In a recent preprint \cite{bgh}, the authors prove that if $S$ is a nonabelian simple group
and $x, y \in \aut(S)$ with $x \neq 1,$ then there exists $s \in S$ such that $\langle x, sy\rangle = \langle S, x, y\rangle.$ This allows immediately to give a complete description of the non-isolated vertices of $\Gamma_2(G).$ Namely if $1\neq x\in G,$ then $x$ is a non-isolated vertex of $\Gamma_2(G)$ if and only if
$\langle x, y, S\rangle=G$ for some $y\in G.$ In particular, if $G/S$ is cyclic, then $1$ is the unique isolated vertex of $\Gamma_2(G)$ and therefore, by \cite[Corollary 6]{spread}, $\Delta_2(G)$ is connected and has diameter at most 2. This reduces the study of the connectivity of $\Delta_2(G)$ to the case when $d(G/S)=2,$ which implies in particular that $S$ is a finite simple group of Lie type. In this case we can prove that $\Delta_2(G)$ is connected with possible finitely many exceptions, using the following recent results.

\begin{thm}[{\cite[Theorem 1.1]{fgg}}]\label{thggf}There exists an absolute constant $\varepsilon > 0$ such that the following holds. Let $S$ be a finite simple group of Lie type of large enough order, and let $x, y \in \aut(S)$ with $x\neq 1.$ Then the probability that $x$ and a random element of $Sy$ generate $\langle S, x,y\rangle$ is at least $\varepsilon.$
\end{thm}

\begin{cor}\label{corfgg} Let $S$ be a finite simple group of Lie type and suppose that $x$ and $y$ are elements of $\aut(S)$ such that $d(\langle x, y, S\rangle/S)=2.$ Consider the graph $\Lambda_{x,y}(S)$ whose parts are the cosets $xS, yS$ and where there is an edge $(xs_1,ys_2)$ if and only if $\langle S, x,y\rangle=\langle xs_1,ys_2\rangle.$ If $|S|$ is large enough, then $\Lambda_{x,y}(S)$ is connected.
\end{cor}

\begin{proof}
	Suppose that $\Lambda_{x,y}(S)$ is not connected. Then there exist two disjoint subsets $A_1, A_2$ of $xS$ and two disjoint subsets $B_1,B_2$ of $yS$ such that $xS=A_1\cup A_2$, $yS=B_1\cup B_2$ and all the edges of $\Lambda_{x,y}(S)$ belong either to $(A_1,B_1)$ or
	to $(A_2,B_2)$. Let $a\in A_1.$ By Theorem \ref{thggf}, there are at least $\varepsilon|S|$ elements of $|B_1|$ adjacent to $a$ in $\Lambda_{x,y}(S),$ hence $|B_1|\geq \varepsilon|S|.$
	A similar argument implies that $|A_1| \geq \varepsilon |S|$, $|A_2| \geq \varepsilon |S|$, and $|B_2| \geq \varepsilon |S|.$ 
	The number of edges of $\Lambda_{x,y}(S)$ is at most
	$$\begin{aligned}|A_1||B_1|+(|S|-|A_1|)(|S|-|B_1|)&=|S|^2-|A_1|(|S|-|B_1|)-|B_1|(|S|-|A_1|)\\&=|S|^2-|A_1||B_2|-|B_1||A_2|
	\leq |S|^2(1-2\varepsilon^2).
	\end{aligned}
	$$
	On the other hand by \cite[Theorem 1.1]{lumo}, the number of edges of $\Lambda_{x,y}$ tends to $|S|^2$ when $|S|$ goes to infinity. So the previous inequality holds only for finitely many choices of $S.$
\end{proof}

We are now able to prove Theorem \ref{almost}.

\begin{proof}[Proof of Theorem \ref{almost}]
	As we noticed before, if $G/S$ is cyclic, then $\Delta_2(G)$ is connected. So we may assume that $S$ is a simple group of Lie type and that $d(G/S)=2.$ Let $x$ and $y$ be two non-isolated vertices of $\Gamma_2(G).$ Then $xS$ and $yS$ are non-isolated vertices of $\Gamma_2(G/S).$ Since $G/S$ is solvable, it follows from \cite[Theorem 1]{CL3} that
	$\Delta_2(G/S)$ is connected. Hence $\Gamma_2(G/S)$ contains a path
	$x_1S=xS-x_2S-\dots-x_rS=yS.$ If $|S|$ is large enough, it follows from Corollary \ref{corfgg} that, for $1\leq i \leq r-1,$ all the elements of $x_iS \cup x_{i+1}S$ belong to the same connected component of $\Gamma_2(G)$. This implies in particular that $x$ and $y$ belong to the same connected component of $\Gamma_2(G).$
\end{proof}

\section{Connectivity of the rank graph}\label{sec:conn_rank}

The proof of Theorem \ref{main_conn_rank} will require some preliminary lemmas.

\

Given a subset $X$ of a finite group $G,$ we will
denote by $d_X(G)$ the smallest cardinality of a set of elements of $G$ generating $G$
together with the elements of $X.$ When $X=\{x\}$ is a singleton, we just write $d_x(G)$ in place of $d_{\{x\}}(G).$
We need to recall an auxiliary result, which generalises an
argument due to Gaschütz \cite{g2}.

\begin{lemma}[{\cite[Lemma 6]{CL3}}] \label{modgg} Let $X$ be a subset of $G$	and $M$ a normal subgroup of $G$ and suppose that	$\langle g_1,\dots,g_r, X, M\rangle=G.$	If $r\geq d_X(G),$ then we can find $n_1,\dots,n_r\in M$ so that $\langle g_1n_1,\dots,g_rn_r,X\rangle=G.$ \end{lemma}

Let $L$ be a finite monolithic group with non-abelian socle $N$. 
The crown-based power of $L$ of size  $k$ is the subgroup $L_k$ of $L^k$ defined by
$$L_k=\{(l_1, \ldots , l_k) \in L^k  \text{ : } l_1 \equiv \cdots \equiv l_k \ {\mbox{mod}\text{ } } N \}.$$

Now we fix a positive integer $t\geq d(L)$ and we denote by $\delta(L,t)$  the maximum integer $k$ \st $L_k$ is $t$-generated (sometimes we write just $\delta$ when there is no possibility of confusion). We introduce some further notation. Let $M=N^k=\soc L_k.$
If $a\in L$ and $m=(n_1,\dots,n_k)\in N^k,$ then we set 
$a\circ m=(an_1,\dots,an_k)\in L_k.$ 

Let $a_1,\dots,a_t\in L$ such that $L=\gen{a_1,a_2,\dots,a_t}$. For any integer $\eta\leq \delta(L,t),$ we define a graph, denoted by
$\Gamma_{a_1,\dots,a_t}(L_\eta),$ whose vertex set is
$$\{a_i \circ m\mid 1\leq i \leq t, m\in N^\eta\},$$
with an edge between $a_i \circ m_i$ and $a_j\circ m_j$ \ifa 
$i\neq j$ and there exist $m_u \in N^\eta$, for $u\notin\{i,j\},$   such that
$$L_\eta=\langle a_i\circ m_i, a_j\circ m_j, a_u\circ  m_u \mid u\notin \{i,j\}\rangle.$$ We denote by $\Delta_{a_1,\dots,a_t}(L_\eta)$ the graph obtained from $ \Gamma_{a_1,\dots,a_t}(L_\eta)$ by removing the isolated vertices. 
\begin{defn} Let $\eta\leq \delta(L,t).$ We say that the
 the graph  $\Delta_{a_1,a_2,\dots,a_t}(L_\eta)$ is weakly connected	if, for $1\leq i\leq t,$ whenever $a_i \circ m_1$ and $a_i\circ m_2$ are two vertices of this graph, there exists $m\in M$ such that $a_i \circ m_1$ and $(a_i\circ m_2)^m$ belong to the same connected component. Moreover we say that crown-based-power $L_\eta$ is $t$-locally connected if the graph  $\Delta_{a_1,a_2,\dots,a_t}(L_\eta)$ is weakly connected whenever   $\langle a_1, a_2,\dots,a_t\rangle=L.$
\end{defn}

We are going to prove a series of intermediate results with the ultimate aim of proving that if $t\geq 3$ then $L_\eta$ is $t$-locally connected for any $\eta\leq \delta(L,t).$ It suffices to concentrate our attention on the particular case when $\eta=\delta(L,t).$ For sake of simplicity, we just write $\delta$ instead of $\delta(L,t).$ Fix $a_1,\dots,a_t\in L$ such that $L=\gen{a_1,a_2,\dots,a_t}$ and let
$$\Omega=\{(a_1n_1,\dots,a_tn_t) \mid n_1,\dots,n_t \in N \text { and } \langle a_1n_1,\dots,a_tn_t\rangle=L\}$$ and let $X:=C_{\aut(L)}(L/N).$ Then $X$ acts on $\Omega$ by setting 
$$(a_1n_1,\dots,a_tn_t)^\gamma=((a_1n_1)^\gamma,\dots,(a_tn_t)^\gamma)$$
and the following holds (see \cite[Section 2]{dv}).
\begin{lemma}\label{primo}Let $\eta\leq \delta$ and let $m_i=(n_{i1},\dots,n_{i\eta}),$ for $1\leq i\leq t.$
	We have $$\langle a_1\circ m_1,\dots, a_t\circ m_t \rangle=L_\eta$$
	\ifa the columns of the matrix
	$$\begin{pmatrix}a_1n_{11}&\cdots& a_1n_{1 \eta}\\a_2n_{21}&\cdots& a_2n_{2\eta}\\\vdots& \cdots &\vdots\\ a_tn_{t1}&\cdots& a_tn_{t \eta}
	\end{pmatrix}$$
	belong to different orbits of $X$ on $\Omega.$
In particular, if $\eta=\delta$ then	
	$$\langle a_1\circ m_1,\dots, a_t\circ m_t \rangle=L_\delta$$
	\ifa the columns of the matrix
	$$\begin{pmatrix}a_1n_{11}&\cdots& a_1n_{1 \delta}\\a_2n_{21}&\cdots& a_2n_{2\delta}\\\vdots& \cdots &\vdots\\ a_tn_{t1}&\cdots& a_tn_{t \delta}
	\end{pmatrix}$$
	give a complete set of representatives for the orbits of $X$ on $\Omega.$
\end{lemma}

Let $M=N^\delta$ and fix $\mu_1,\dots,\mu_t \in M$ such that $L_\delta=\langle a_1 \circ \mu_1,\dots, a_t \circ \mu_t\rangle.$ Set $\mu_i=(c_{i1},\dots,c_{i\delta}).$ The following is a consequence of Lemma \ref{primo}.
\begin{lemma}\label{edge}
	Recall that $X:=C_{\aut(L)}(L/N)$ and let $K:=X\wr \sym(\delta)\leq \aut(L_\delta)$. If 
	$\langle a_1\circ m_1,\dots, a_t\circ m_t\rangle=L_\delta,$ then there exists 
	an element $\gamma$ in $K$ \st 
	$a_i\circ m_i=(a_i\circ \mu_i)^\gamma,$ for $1\leq i\leq t.$
\end{lemma}


To every $1\leq i \leq t,$ we may associate a partition $\pi_i=\{\Omega_{i,1},\dots,\Omega_{i,r_i}\}$ of the set $\{1,\dots,\delta\}$ by setting that $j_1$ and $j_2$ belong to the same part if and only if
$a_ic_{i,j_2}=(a_ic_{i,j_1})^\gamma$ for some $\gamma \in X$. 
The set $\mathcal P$ of partitions of $\{1,\dots,\delta\},$ ordered by refinement, is a lattice with smallest element $\pi_{\text{min}}=\{\{1,\dots,\delta\}\}$ and maximal element $\pi_{\text{max}}=\{\{1\},\dots,\{\delta\}\}.$

\begin{lemma}\label{cyct}
	Suppose that the following hold:
	\begin{enumerate}\item  $\pi_{\text{min}}=\pi_1\wedge \dots \wedge \pi_t$,
		\item for any $\gamma\in X,$ there exist $i \in \{1,\dots,t\},$ 
		$j \in \{1,\dots,\delta\}$, $\beta\in X$ and $n\in N$
		such that $[a_ic_{ij},\gamma^\beta n]=1.$
	\end{enumerate} Then $\Delta_{a_1,\dots,a_t}(L_\delta)$
	is weakly connected.
\end{lemma}

\begin{proof}
	We want to prove that, for $1\leq i \leq t,$ if $a_i\circ \tilde m$ is a vertex of $\Delta,$ then there exists $m\in M$ such that
	$(a_i\circ\tilde m)^m$ and $a_i \circ \mu_i$ belong to the same connected component. It suffices to prove this statement for $i=1.$ Let $\Lambda$ be the connected component of $\Delta=\Delta_{a_1,\dots,a_t}(L_\delta)$ containing
	$a_1\circ \mu_1$ and let $H$ be the setwise stabilizer
	of $\Lambda$ in $K$ and let $J=HM.$ Then $\langle C_1,\dots,C_t\rangle \leq H\leq J,$ 
	where $C_i=C_K(a_i\circ \mu_i).$ For every part $\Omega_{ij}$ of $\pi_i$, we can fix $a_i\nu_{ij} \in a_iN$ such that, for $k\in\Omega_{ij}$, $a_ic_{ik}$ is $X$-conjugate to $a_i\nu_{ij}.$ It turns out that there exists $\psi_i \in X^\delta$ such that
	$$C_i=\left(\prod_{1\leq j\leq r_i}C_X(a_i\nu_{ij})\wr \sym(\Omega_{ij})\right)^{\psi_i}.$$
	Consider the natural projection $\rho\colon K\to\sym(\delta)$. The assumption $\pi_{\text{min}}=\pi_1\wedge \dots \wedge \pi_t$ ensures
	that $$\langle C_1^\rho, \dots  C_t^\rho\rangle=\left\langle \prod_{1\leq j\leq r_1} \sym(\Omega_{1j}),\dots,\prod_{1\leq j\leq r_t} \sym(\Omega_{tj})\right\rangle$$ is a transitive subgroup of $\sym(\delta).$
	As a consequence, $H^\rho$ is transitive too.

	Let $B$ be a block of the action on $H^\rho$ and assume $|B|>1.$ 
	If $\Omega_{ij}\cap B\neq\emptyset,$ then $\Omega_{ij}\subseteq B$, so, again from
	$\pi_{\text{min}}=\pi_1\wedge \dots \wedge \pi_t,$ we deduce that $B=\{1,\dots,\delta\}$ and so $H^\rho$ is a primitive subgroup of $\sym(\delta)$. If $\delta \neq 1,$ then  the condition $\pi_{\text{min}}=\pi_1\wedge \dots \wedge \pi_t$ implies that
	there exists $i\in \{1,\dots,t\}$ such that $\pi_i \neq \pi_{\text{max}},$ hence $|\Omega_{i,j}|\geq 2$ for some $1\leq j\leq r_i$. Thus $H^\rho$
	contains a transposition, and therefore by \cite[Theorem 7.4B]{DM} we conclude $H^\rho=\sym(\delta)$. Of course this implies $J^\rho=\sym(\delta)$.
	
	For any $1\leq j\leq \delta$ and $\gamma \in X,$ let $\tau_{j,\gamma}:=(1,\dots,1,\gamma,1,\dots,1)\in X^\delta,$ with $\gamma$ in the $j$-th position.
If $\gamma \in X$, then $[a_ic_{ij},\gamma^\beta n]=1$ for some $1\leq i\leq t,$ $1\leq j\leq \delta$ and $\beta\in X$.   Since $[a_i\circ \mu_i,\tau_{j,\gamma^\beta n}]=1,$ we have $\tau_{j,\gamma^\beta n}\in H.$ It follows that $\tau_{j,\gamma^\beta}\in J=HM$ for every $\gamma \in X$ and some $\beta\in X$. Since 
$J^\rho$ is transitive, we conclude that for each $\gamma\in X$, there exists $\beta'\in X$ \st $\tau_{1,\gamma^{\beta'}}\in J$. If $X_1:=\{(x,1,\dots,1)\mid x\in X\}\cong X$, we have that $X_1\cap J$ is a subgroup of $X$ which contains a representative for each conjugacy class of $X$, and therefore $X_1\cap J=X_1$ and by the transitivity of $J^\rho$, we conclude that $X^\delta$ is contained in $J$ and therefore $K=J=HM.$ Now assume that $a_1\circ \tilde m$ is a vertex of $\Delta.$ By Lemma \ref{edge}, there exists $k\in K$ such that $a_1\circ \tilde m=(a_1\circ \mu_1)^k.$ Write $k=hm$ with
$h\in H$ and $m\in M.$ Then $(a_1\circ \tilde m)^{m^{-1}}=(a_1\circ \mu_1)^h\in \Lambda.$ 
\end{proof}



The previous lemma can be improved if $t\geq 3.$ For this purpose we need the following prelimary result.

\begin{lemma}\label{delu} Let $l \in L$ and suppose $d\geq \max\{2,d_l(L)\}.$	If $L=\langle l, b_1,\dots, b_d\rangle$, then the set	$\Omega(l;b_1,\dots,b_d)$ of the elements $(n_1,\dots,n_d)\in N^d$ such that $L=\langle l, b_1n_1,\dots,b_dn_d\rangle$ has cardinality at least $53|N|^d/90.$\end{lemma}\begin{proof}	The statement is a particular case of \cite[Theorem 2.2]{dl}.\end{proof}

\begin{lemma}\label{sempreuno}If $t\geq \max\{3,d(L)\},$ then condition (1) in Lemma \ref{cyct} is satisfied.\end{lemma}
\begin{proof}	We may assume $1 \in \Omega_{1,1}$. 
	If $\pi_1=\pi_{\text{min}},$ then clearly
	$\pi_{\text{min}}=\pi_1\wedge \dots \wedge \pi_t$. So we may assume $\pi_1=\{\Omega_{1,1},\dots,\Omega_{1,r_1}\}$ with $r_1>1.$ To prove our statement it suffices to show that for any  $2\leq k\leq r_1,$ there exists $1\leq  u\leq r_2$ such that $\Omega_{2,u} \cap \Omega_{1,1}$ and $\Omega_{2,u} \cap \Omega_{1,k}$ are both non-empty.
	Let $j \in\Omega_{1,k}$ with $k > 1.$ It follows from Lemma \ref{delu} that $\Omega(a_1c_{1,1};a_2,\dots,a_t)\cap \Omega(a_1c_{1,j};a_2,\dots,a_t)\neq \emptyset,$ so there exist $\nu_2,\dots,\nu_t\in N$ such that $\langle a_1c_{1,1}, a_2\nu_2,\dots,a_t\nu_t \rangle= \langle a_1c_{1,j}, a_2\nu_2,\dots,a_t\nu_t \rangle=L_\delta$. It follows from Lemma \ref{primo} that there exists $1\leq u< v\leq \delta$ and $\gamma_u,\gamma_v \in X$ such that
	$$\begin{aligned}(a_1c_{1,u},a_2c_{2,u},\dots,a_tc_{t,u})&=(a_1c_{1,1}, a_2\nu_2,\dots,a_t\nu_t)^{\gamma_u},\\
		(a_1c_{1,v},a_2c_{2,v},\dots,a_tc_{t,v})&=(a_1c_{1,j}, a_2\nu_2,\dots,a_t\nu_t)^{\gamma_v}.
		\end{aligned}$$
But then $\Omega_{2,u}=\Omega_{2,v}$. Since $u\in \Omega_{1,1}$ and $v\in \Omega_{1,j},$ we are done. \end{proof}

\begin{lemma}\label{unico_rank}
	Let $t\geq 3$ and suppose $L=\langle b_1,\dots,b_t\rangle N.$ Then
	$d_{b_t}(L)\leq t-1.$
\end{lemma}
\begin{proof}
	We may identify $L$  with a subgroup of $\aut(S^\delta)=\aut S \wr \sym (\delta)$. 
	Let $\pi: \aut S \wr \sym (\delta) \to \sym (\delta)$ be the homomorphism which maps
	$(\alpha_1, \dots, \alpha_\delta)\sigma$ to $\sigma$. Suppose $L=\langle b_1,\dots,b_t\rangle.$ We may choose $b_1$ and $b_2$ in  such a way that
	$\pi(b_1)$ is not a cycle of length $\delta$; moreover if
	$\pi(b_1)$ has no fixed points but there exist
	$\bar b_1,  \bar b_2 \in L$ such that $L=\langle \bar b_1,\bar b_2,b_3,  \dots,
	b_t\rangle$ and $\pi(\bar l_1)$ has a fixed point, then we substitute
	$b_1,b_2$  by  $\bar b_1, \bar b_2.$
	With this choice of $b_1,\dots,b_t,$
	the proof of Theorem 1.1 in \cite{LM} shows that  there exist
	 $n_1, n_2 \in N$ such that
	$L=\langle b_1n_1, b_2n_2, b_3\dots, b_t \rangle$.
\end{proof}
\begin{lemma}\label{condizione}
	If $t\geq 3,$ then we may replace condition (2) in Lemma \ref{cyct} with the following condition:
	
	(3) for any $\gamma\in X,$ there exist  $m, n\in N$
	such that $[a_tm,\gamma n]=1.$
\end{lemma}
\begin{proof}
	Suppose that the new condition holds. We want to prove that for any $\gamma\in X,$ there exist  $i \in \{1,\dots,t\},$ 
	$j \in \{1,\dots,\delta\}$, $\beta\in X$ and $n\in N$
	such that $[a_ic_{ij},\gamma^\beta n]=1.$	Choose $m, n \in N$ such that $[a_tm,\gamma n]=1.$ Since $L=\langle a_1,\dots,a_t\rangle=\langle a_1,\dots,a_{t-1},a_tm\rangle N,$ by Lemma \ref{unico_rank} we have that $d_{a_tm}(L)\leq t-1,$ and it follows from Lemma \ref{modgg}  that
	there exist $n_1,\dots, n_{t-1} \in N$ such that
	$L=\langle a_1n_1,\dots a_{t-1}n_{t-1}, a_tm \rangle$. This implies that there exists $j \in \{1,\dots,\delta\}$ and $\beta\in X$ such that
	$a_tc_{tj}=(a_tm)^\beta.$ But then $[a_tm,\gamma n]^\beta=[a_tc_{tj},(\gamma n)^\beta]=1.$ 
	Since $(\gamma n)^\beta=\gamma \tilde n$ for some $\tilde n\in N,$ we conclude that $[a_tc_{tj},\gamma^\beta \tilde n]=1$ and then
	condition (2) in Lemma \ref{cyct} is satisfied.
\end{proof}

\begin{prop}\label{fatica}
	If $t\geq 3,$ then $\Delta_{a_1,\dots,a_t}(L_\delta)$ is weakly connected. 
\end{prop}

\begin{proof}
	Combining Lemmas \ref{cyct}, \ref{sempreuno} and \ref{condizione}, it suffices to prove that for any $\gamma\in X,$ there exist  $m, n\in N$
	such that $[a_tm,\gamma n]=1.$ Since, by definition, $[a_t,\gamma]\in N,$ the conclusion follows from Theorem \ref{cln}.
\end{proof}

\begin{cor}\label{semprelocallycon}
	If $t\geq 3,$ then $L_\eta$ is $t$-locally connected for any $\eta\leq \delta(L,t).$
\end{cor}

\begin{proof}
We have a map $\xi: N^\delta\to N^\eta$ sending $(n_1,\dots,n_\delta)$ to $(n_1,\dots,n_\eta).$ 
Assume $L=\langle a_1,\dots,a_t\rangle$. It follows from
Lemma \ref{primo} that if $a_i\circ m$ is a non-isolated vertex of the graph $\Gamma_{a_1,a_2,\dots,a_t}(L_\eta),$ then there exists
$\tilde m \in N^\delta$ such that $m=\tilde m^\xi$ and $a_i\circ \tilde m$ is a non-isolated vertex of  $\Gamma_{a_1,a_2,\dots,a_t}(L_\delta).$ In particular if
$a_i\circ m_1$ and $a_i\circ m_2$ are non-isolated vertices of $\Gamma_{a_1,a_2,\dots,a_t}(L_\eta),$ then, by Proposition \ref{fatica}, in 
$\Gamma_{a_1,a_2,\dots,a_t}(L_\delta)$ there is a path joining 
$a_i\circ \tilde m_1$ and $(a_i\circ \tilde  m_2)^\nu$ for some $\nu\in N^\delta$. By applying $\xi,$ we obtain a path in $\Gamma_{a_1,a_2,\dots,a_t}(L_\eta)$ from $a_i\circ m_1$ to $(a_i\circ m_2)^{\nu^\xi}.$
\end{proof}

The following easy result plays a crucial role in the study of the connectivity properties of the graph $\Gamma_d(G).$

\begin{lemma}\label{frat}
	If $\Gamma_d(G/\frat(G))$ is connected, then $\Gamma_d(G)$ is connected.
\end{lemma} 

\begin{proof}
	It follows immediately from the fact that $G=
	\langle x_1,\dots,x_d\rangle$ if and only if
		$G/\frat(G)=
	\langle x_1\frat(G),\dots,x_d\frat(G)\rangle.$
\end{proof}

\begin{lemma}\label{induzionenormale} Let $M$ be a normal subgroup of a finite group $G$ and assume that $x$ and $y$ are two non-isolated vertices of $\Gamma_d(G).$
	If $xM$ and $yM$ are in the same connected component of $\Gamma_d(G/M),$ then there  exists $m\in M$  such that $x$ and $ym$ are in the same connected component of $\Gamma_d(G).$
\end{lemma}
\begin{proof}Suppose that $g_1M - g_2M - \dots - g_tM$ is a path in $\Gamma_d(G/M)$ with $g_1=x$ and $g_t=y.$ To conclude the proof of the claim, it suffices to prove, by induction on $j,$ that for each $1 \leq j \leq t,$ there exists $m_j\in M$ such that $\Gamma_d(G)$ contains a path joining $x$ and
	$g_jm_j.$ Assume that the statement is true for $j<t.$ Since $g_jM$ and $g_{j+1}M$ are adjacent in $\Gamma_d(G/M)$, there exist $h_3,\dots,h_d$ such that $\langle g_jm_j, g_{j+1}, h_3,\dots,h_d\rangle M=G.$ 
	By Lemma \ref{modgg},  $\langle g_jm_j, g_{j+1}z_2, h_3z_3,\dots,h_dz_d\rangle=G,$ for suitable $z_2,\dots,z_d\in M.$ Taking $m_{j+1}=z_2,$ we can conclude that the statement is true for $j+1.$
\end{proof}

\begin{lemma}\label{norsol}
	If $N$ is a soluble normal subgroup of $G$ and $\Delta_d(G/N)$ is connected, then $\Delta_d(G)$ is also connected.
\end{lemma}

\begin{proof}
	We prove our statement by induction on the order of $G.$ We may assume that $N$ is a minimal normal subgroup of $G$ and that it is not contained in the Frattini subgroup of $G.$ Thus there exists a complement $H$ of $N$ in $G$ and $\Delta_d(H)\cong \Delta_d(G/N)$ is connected.

	It follows from Lemma \ref{induzionenormale} that in order to prove that $\Delta_d(G)$ is connected, it suffices to prove that for any vertex $h$ of $\Delta_d(H)$, all the vertices of $\Delta_d(G)$ contained in the coset $Nh$ belong to the same connected component of $\Delta_d(G).$

	So suppose by contradiction that the previous statement fails for a given non-isolated vertex $h$ of $\Gamma_d(H).$ Let $F:=\End_H(N)$, $f:=|F|$ and $\dim_F N=u$. Assume $H=\langle h, h_2,\dots,h_d\rangle$ and, for any $n\in N,$ let
	$$\Omega_{nh}=\{(n_2,...,n_d)\in N^{d-1}\mid \langle nh, n_2h_2\dots,n_dh_d \rangle \neq G\}.$$
	There must exist $n_1, n_2 \in N$ such that $n_1h$ and $n_2h$ belong to different components of $\Gamma_d(G).$ This implies $\Omega_{n_1h}\cup \Omega_{n_2h}=N^{d-1}$ and therefore there exists $n\in \{n_1,n_2\}$ with $|\Omega_{nh}|\geq |N|^{d-1}/2.$
	Notice that $|\Omega_{nh}|$ coincides with the number of complements
	 of $N$ in $G$ containing $nh$ and therefore $|\Omega_{nh}|=f^k,$ where $k$ is the dimension over $F$ of the group ${\rm{Der}}(K,N)$ of derivations from $K$ to $N$ which fix $nh,$ being $K$ one of these complements. Since $\Omega_{nh}\neq N^{d-1},$ we have
	$$	\frac{f^{u(d-1)}}{2}\leq f^k < f^{u(d-1)},
	$$
	which is possible only if $f=2$ and $k=u(d-1)-1$. This implies also $|\Omega_{n_1h}|=|\Omega_{n_2h}|=f^k.$ 
	Notice that $f=2$ implies that $N$ has exponent 2.
	Let $c:=f^k$. Assume that
	$v_1,\dots,v_c$ are the elements of $\Omega_{n_1h}$ and
	$w_1,\dots,w_c$ the elements of $\Omega_{n_2h}$.
	For $1\leq i \leq c,$ we use the notation
	$$v_i=(v_{2i},\dots,v_{di}), w_i=(w_{2i},\dots,w_{di}),$$
	with $v_{ji}, w_{ji} \in N.$ Notice that, for 
	$1\leq i\leq c,$ we must have
	$$\langle n_1h, w_{2j}h_2,\dots,w_{dj}h_d\rangle
	=\langle n_2h, v_{2j}h_2,\dots,v_{dj}h_d\rangle=G.$$
	In particular $v_{21}h_2$ is a non-isolated vertex of $\Gamma_d(G)$
	and belongs to the connected component of $\Gamma_d(G)$ containing $n_2h.$ Moreover $n_1h$ and $(n_1h)(w_{2j}h_2)$ are adjacent vertices of $\Delta_d(G),$ and therefore  $v_{21}h_2$ and
	$n_1hw_{2j}h_2$ belong to different components of $\Delta_d(G).$
	This implies that for, any $1\leq j\leq c,$
	$$\langle n_1hw_{2j}v_{21}, v_{21}h_2,w_{3j}v_{31}h_3,\dots, w_{dj}v_{d1}h_d\rangle$$ is a complement of $N$ in $G.$
	Since the $c$ elements $(w_{2j}v_{21},\dots,w_{dj}v_{d1})$ of
	$N^{d-1}$ are all distinct, we have constructed $c$ different complements of $N$ in $G$ which contain $v_{21}h_2.$ Moreover
	$w_{2j}v_{21}\neq 1,$ otherwise $n_1h$ and $n_2h$ would be adjacent to $v_{21}h_2=w_{2j}h_2$, and therefore
	$\langle n_1h, v_{21}h_2,v_{31}h_d\dots, v_{d1}h_d\rangle$
	is another complement containing $v_{21}h_2$ which is different from the previous ones. We have so proved that there are at least $c+1=f^k+1$ complements of $N$ in $G$ containing $v_{21}h_2.$ Since the number of these complements is a power of $f$, we deduce that $\langle z_1h, v_{21}h_2,z_3h_3,\dots,z_dh_d\rangle$  is a complement of $N$ in $G$ for every $(z_1, z_3, \dots, z_d) \in N^{d-1},$ in contradiction with the fact that $\langle n_2h, v_{21}h_2,\dots,v_{d1}h_d\rangle=G.$
\end{proof}

To conclude our proof, we need to recall some properties of the crowns of finite groups. The notion of crown was
introduced by Gasch\"{u}tz  in \cite{Gaschutz} in the case of finite soluble groups and generalised
in \cite{JL} to arbitrary finite groups.
A detailed exposition of the theory is also given in \cite[Section 1.3]{cl}. If a group $G$ acts on a group $A$ via automorphisms (that is, if there exists a homomorphism $G\rightarrow \aut(A)$), then we say that $A$ is a $G$-group. If $G$ does not stabilise any nontrivial proper subgroup of $A$, then $A$ is called an irreducible $G$-group. Two $G$-groups $A$ and $B$ are said to be $G$-isomorphic, or $A\cong_G B$, if there exists a group isomorphism $\phi: A\rightarrow B$ such that 
$\phi(g(a))=g(\phi(a))$ for all $a\in A, g\in G$.  Following  \cite{JL}, we say that two  $G$-groups $A$ and $B$  are $G$-equivalent and we put $A \equiv_G B$, if there are isomorphisms $\phi: A\rightarrow B$ and $\Phi: A\rtimes G \rightarrow B \rtimes G$ such that the following diagram commutes:

\begin{equation*}
	\begin{CD}
		1@>>>A@>>>A\rtimes G@>>>G@>>>1\\
		@. @VV{\phi}V @VV{\Phi}V @|\\
		1@>>>B@>>>B\rtimes G@>>>G@>>>1.
	\end{CD}
\end{equation*}

\

Note that two  $G$\nobreakdash-isomorphic $G$\nobreakdash-groups are $G$\nobreakdash-equivalent. In the particular case where $A$ and $B$ are abelian the converse is true: if $A$ and $B$ are abelian and $G$\nobreakdash-equivalent, then $A$ and $B$ are also $G$\nobreakdash-isomorphic.
It is proved (see for example \cite[Proposition 1.4]{JL}) that two  chief factors $A$ and $B$ of $G$ are $G$-equivalent if and only if either they are  $G$-isomorphic, or there exists a maximal subgroup $M$ of $G$ such that $G/\core_G(M)$ has two minimal normal subgroups $X$ and $Y$ that are
$G$-isomorphic to $A$ and $B$ respectively. For example, the minimal normal subgroups of a crown-based power $L_k$ are all $L_k$-equivalent.

For an irreducible $G$-group $A$, we define $\delta_G(A)$ to be the number of non-Frattini
chief factors $G$-equivalent to $A$ in a chief series for $G$ and we denote by $L_A$ the monolithic primitive group associated to $A$.
That is
$$L_{A}=
\begin{cases}
	A\rtimes (G/C_G(A)) & \text{ if $A$ is abelian}, \\
	G/C_G(A)& \text{ otherwise}.
\end{cases}
$$
If $A$ is a non-Frattini chief factor of $G$, then $L_A$ is a homomorphic image of $G$. More precisely, there exists a normal subgroup $N$ of $G$ such that $G/N \cong L_A$ and $\soc(G/N)\equiv_G A$. Consider now all the normal subgroups $N$ of $G$ with the property that $G/N \cong L_A$ and $\soc(G/N)\equiv_G A$:
the intersection $R_G(A)$ of all these subgroups has the property that  $G/R_G(A)$ is isomorphic to the crown-based power $(L_A)_{\delta_G(A)}$. The socle $I_G(A)/R_G(A)$ of $G/R_G(A)$ is called the {$A$-crown} of $G$ and it is  a direct product of $\delta_G(A)$ minimal normal subgroups $G$-equivalent to $A$.

\begin{lemma}\label{coniugo}
	Let $G$ be a finite group and $g$ a vertex of $\Delta_d(G).$ 
	Then the connected component of $\Delta_d(G)$ containing $g$, contains also $g^z$ for every $z \in G.$
\end{lemma}
\begin{proof}
	We use the symbol $x \sim y$ to denote that either $x=y$ or there is a path in $\Delta$ joining $x$ and $y.$ Notice that if $x\sim x^a$ and $x\sim x^b$, then by conjugating by $b$ the first relation we get $x^b\sim x^{ab}$, so $x\sim x^b\sim x^{ab}$, and therefore $x\sim x^{ab}.$ Moreover if $x\sim x^a$, then by conjugating by $a^{-1}$, we get $x\sim x^{a^{-1}}$. This means that $G_g:=\{a\in G \mid g\sim g^a\}$	is a subgroup of $G.$ Since $g$ is a vertex of $\Delta_d(G)$, there exist $g_2,\dots,g_d$ such that  $G=\langle g,g_2,\dots,g_d\rangle=\langle g^{g_i},g_2^{g_i},\dots,g_d^{g_i}\rangle.$ Since $g_i^{g_i}=g_i$, we have that $g_i$ is adjacent to both $g$ and $g^{g_i},$
	so $g \sim g^{g_i}.$ But then $G=\langle g_1,\dots,g_d\rangle \leq G_g.$ 
\end{proof}

\begin{lemma}\label{riduco}
	Suppose that every non-abelian crown-based-power appearing as epimorphic image of a finite group $G$ is $d$-locally connected. Then $\Delta_d(G)$ is connected.
\end{lemma}
\begin{proof}
	We prove the theorem by induction on the order of $G.$ 
By Lemma \ref{frat}, we may assume $\frat(G)=1.$ By \cite[Lemma 1.3.6]{cl}, there exists
	a crown $I_G(N)/R_G(N)$ and a non-trivial normal subgroup $U$ of $G$ such that $I_G(N)=R_G(N)\times U.$ By Lemma \ref{norsol} we may assume that $U$ is non-abelian. We have that $N$ is a $G$-group and $U \sim_G N^\delta,$ for $\delta=\delta_G(N)$ and there exists an epimorphism $\phi: G  \to L_\delta$ where $L$ is a finite monolithic group with $\soc(L)\cong N,$
	such that $\ker \phi = R_G(N)$ and $U^\phi=N^\delta.$
	
	Choose two non-isolated vertices $x$ and $y$ of $\Gamma_d(G).$ By Lemma \ref{induzionenormale}, there exists $u \in U$ such that $\Gamma_d(G)$ contains a path joining $x$ to $yu.$
	Therefore, to conclude our proof it suffices to prove that 
$y$ and $yu$ belong to the same connected component of $G.$ Suppose $\langle y, y_2,\dots,y_d\rangle=G$.
	By hypothesis, the graph $\Delta_{y^\phi,y_2^\phi,\dots,y_d^\phi}(G^\phi)$ is weakly connected. Hence, 
	there exist a positive integer $v$ and $\tilde u,$ $u_{i,j} \in U,$ with $1\leq i \leq d$ and $1\leq j \leq v,$ such that the subsets $\Omega_j=\{yu_{1j},y_2 u_{2j},\dots,y_du_{dj}\}$ satisfy the following properties:
	\begin{enumerate} 
		\item for $1\leq j\leq v$, $\langle\Omega_j\rangle^\phi= G^\phi$;
		\item for $1\leq j\leq v-1,$ $\Omega_j \cap \Omega_{j+1} \neq \emptyset;$ 
		\item $y \in \Omega_1$ and $(yu)^{\tilde u} \in \Omega_v.$
	\end{enumerate}
	Since $\langle \Omega_j \rangle^\phi=G^\phi$ and $R_G(N)=\ker \phi,$ $\langle \Omega_j \rangle R_G(N)=G.$ On the other hand, $\langle \Omega_j \rangle U=
	\langle y,\dots,y_d\rangle U=G.$ By \cite[Proposition 11]{crowns}, we conclude
	$\langle \Omega_j\rangle=G$, and therefore we can construct a path in $\Gamma_d(G)$ between $y$ and $(yu)^{\tilde u}.$ To conclude it is enough to notice that, by Lemma \ref{coniugo},
	$yu$ and $(yu)^{\tilde u}$ belong to the same connected component of $\Gamma_d(G).$
\end{proof}

%
%

We are now able to conclude the proof.
\begin{proof}[Proof of Theorem \ref{main_conn_rank}]
Suppose that $L_\eta$ is a crown-based power which appears as an epimorphic image of $G.$ Then $\eta\leq \delta(L,d)$ and, since $d\geq 3,$ $L_\eta$ is $d$-locally connected by Corollary \ref{semprelocallycon}. Hence the conclusion follows from Lemma \ref{riduco}.
\end{proof}

\end{document}